\documentclass[12 pt]{article}%
\usepackage{amsmath, amsfonts, amsthm, color,latexsym}
\usepackage{amsmath}
\usepackage{amsfonts}
\usepackage{amssymb}
\usepackage{color, soul}
\usepackage[all]{xy}
\usepackage{graphicx}%
\providecommand{\U}[1]{\protect\rule{.1in}{.1in}}
\allowdisplaybreaks[4]
\newtheorem{theorem}{Theorem}[section]
\newtheorem{proposition}[theorem]{Proposition}
\newtheorem{corollary}[theorem]{Corollary}
\newtheorem{example}[theorem]{Example}

\newtheorem{lemma}[theorem]{Lemma}
\newtheorem{final remark}[theorem]{Final Remark}
\newtheorem{definition}[theorem]{Definition}
\allowdisplaybreaks[4]

\begin{document}

\title{The positive polynomial Schur property in Banach lattices}
\author{Geraldo Botelho\thanks{Supported by CNPq Grant
304262/2018-8 and Fapemig Grant PPM-00450-17.}\,\, and  Jos\'e Lucas P. Luiz\thanks{Supported by a CNPq scholarship\newline 2020 Mathematics Subject Classification: 46B42, 46G25, 46A40 .\newline Keywords: Banach lattices, positive Schur property, regular polynomials, weak Dunford-Pettis property, positive polynomial Schur property.
}}
\date{}
\maketitle

\begin{abstract} We study the class of Banach lattices that are positively polynomially Schur. Plenty of examples and counterexamples are provided, lattice properties of this class are proved, arbitrary $L_p(\mu)$-spaces, $1 \leq p < \infty$, are shown to be positively polynomially Schur, lattice analogues of results on Banach spaces are obtained and relationships with the positive Schur and the weak Dunford-Pettis properties are established.
\end{abstract}

\section{Introduction}

In the realm of Banach spaces, the polynomial Schur property (polynomially null sequences are norm null) was introduced by Carne, Cole and Gamelin \cite{carne} and developed by several authors (see, e.g., \cite{aron, arranz, farmer,jaramillo}). Banach spaces with this property are called $\Lambda$-spaces or polynomially Schur spaces. Of course this property is related to the Schur property (weakly null sequences are norm null). It is clear that Banach spaces with the Schur property are polynomially Schur and it is well known that a Banach space has the Schur property if and only if it is polynomially Schur and has the Dunford-Pettis property (see \cite[Theorem 7.5]{carne}).

In the setting of Banach lattices, the positive Schur property (weakly null sequences formed by positive vectors are norm null) has been extensively studied, recent developments can be found in \cite{arkadani, baklouti, botelho, chen, tradacete, wnuk4, zeekoei}. So it is a natural step to consider the lattice counterpart of the polynomial Schur property, which is the main subject of this paper.

\begin{definition}\rm We say that a Banach lattice $E$ has the {\it positive polynomial Schur property} if every sequence $(x_j)_{j=1}^\infty$ in $E$ formed by positive vectors such that $P(x_j) \longrightarrow 0$ for every regular homogeneous polynomial $P$ on $E$ is norm null.
\end{definition}

A Banach lattice with the positive polynomial Schur property shall be called a {\it PPS lattice} (positively polynomially Schur lattice).

In Section 2 we give several examples and counterexamples, showing that the class of PPS lattices is neither very small nor very large, in the sense that good examples are available and nice general properties can be found. Some of this properties, especially lattice properties, are obtained in this same section. Section 3 is devoted to the proof that arbitrary $L_p(\mu)$-spaces, $1 \leq p < \infty$, are PPS lattices and to consequences of this fact. In Section 4 we establish the interplay between the positive polynomial Schur and the weak Dunford-Pettis properties. We prove the lattice analogue of the equivalence stated in the first paragraph of this Introduction for Banach spaces. Some consequences are derived.

By $E^*$ we denote the topological dual of the Banach space $E$. Remember that a homogeneous polynomial on a Riesz space is {\it positive} if its associated symmetric multilinear form is positive, and that the difference of two positive homogeneous  polynomials is called a {\it regular} homogeneous polynomial. 
By $\mathcal{P}^r(^nE)$ we denote the set of regular $n$-homogeneous (scalar-valued) polynomials on the Banach lattice $E$, which becomes a Banach lattice with the regular norm
$$\|P\|_r = \| |P| \|, $$
where $|P|$ is the absolute value of the polynomial $P$.
Recall that every polynomial in $\mathcal{P}^r(^nE)$ is continuous  (see \cite[Proposition 4.1]{grecu}).
For the theory of regular homogeneous polynomials on Banach lattices we refer to \cite{bu, loane1}. 

\section{Basic properties and examples}

We start with some examples.

\begin{example}\label{L1Plam} \rm (a) Every Banach lattice with the positive Schur property is a PPS lattice. Indeed, this follows immediately from the fact that bounded linear functionals are regular, that is, $E^*=\mathcal{P}^r(^1E)$. Since every AL-space has the positive Schur property \cite[Example 1.3(a), p.\,161]{rabiger}, it follows that every AL-space is a PPS lattice. In particular, $L_1(\mu)$ is a PPS lattice for any measure $\mu$.\\
(b) $L_1[0,1]$ is a PPS lattice that fails to be a polynomially Schur Banach space (see \cite[Theorem 6.5]{carne}).\\
(c) Not only AL-spaces are PPS lattices: it follows from what was shown in \cite[Example 2.8]{primeiro} that the Banach lattice $\big(\oplus_{n=1}^\infty\ell^n_\infty \big)_1$ has the positive Schur property, so it is a PPS lattice, and that it is not an AL-space.
\end{example}

In Section 3 we shall provide examples of infinite dimensional reflexive PPS lattices, establishing, in particular, that not only Banach lattices with the positive Schur property are PPS lattices. Next result gives a lot of counterexamples. Recall that a Banach space $E$ has the Dunford-Pettis property if all weakly compact operators defined on $E$ are completely continuous.

\begin{proposition}\label{fpro} Banach lattices with the Dunford-Pettis property and lacking the positive Schur property are not PPS lattices. In particular, {\rm AM}-spaces are not PPS lattices.
\end{proposition}

\begin{proof} Let $E$ be a such Banach lattice and let $(x_j)_{j=1}^\infty$ be a positive weakly null non-norm null sequence in $E$. Continuous homogeneous polynomial on $E$ are weakly sequentially continuous by \cite[Proposition 2.34]{dineen}, so $ P(x_j) \longrightarrow 0$ for all $P \in {\cal P}^r(^nE)$ and $n \in \mathbb{N}$. It follows that $E$ is not a PPS lattice. AM-spaces contain lattice copies of $c_0$ \cite[Exercise 16, p.\,254]{aliprantis2}, so they lack the positive Schur property, and they have the Dunford-Pettis property \cite[Proposition 3.7.9]{meyer}.
\end{proof}

\begin{example}\label{exAM}\rm $C(K)$-spaces, in particular $c_0$, are not PPS lattices.
\end{example}

To check that a Banach lattice has the positive Schur property, it is enough to consider positive {\it disjoint} weakly null sequences (see \cite[p.\,16]{wnuk}). Next we show that the same holds for the positive polynomial Schur property.

\begin{proposition} A Banach lattice $E$ is a PPS lattice if and only if every disjoint sequence $(x_j)_{j=1}^\infty$ in $E$ formed by positive vectors such that $P(x_j) \longrightarrow 0$ for every regular homogeneous polynomial $P$ on $E$ is norm null.
\end{proposition}

\begin{proof} One direction is obvious. To prove the converse, assume that every disjoint sequence $(x_j)_{j=1}^\infty$ in $E$ formed by positive vectors such that $P(x_j) \longrightarrow 0$ for every regular homogeneous polynomial $P$ on $E$ is norm null. Suppose that $E$ fails to be a PPS lattice, that is, there exists a non-norm null sequence $(x_j)_{j=1}^\infty$ in $E$ formed by positive vectors such that $P(x_j) \longrightarrow 0$ for every regular homogeneous polynomial $P$ on $E$. In this case there exists $\varepsilon>0$ and a subsequence $(x_{j_k})_{k=1}^\infty$ of $(x_j)_{j=1}^\infty$ such that $\|x_{j_k}\|\geq\varepsilon$ for every $k\in\mathbb{N}$. Calling $z_k:=(1/\varepsilon) x_{j_k}$ for every $k$, the sequence $(z_k)_{k=1}^\infty$ is formed by positive vectors, $\|z_k\|\geq 1$ for each $k\in\mathbb{N}$ and $P(z_k) = (1/\varepsilon^n) P(x_{j_k})\longrightarrow 0$ for all $P\in\mathcal{P}^r(^nE)$ and $n\in\mathbb{N}$. Since $E^*=\mathcal{P}^r(^1E)$, in particular the sequence $(z_k)_{k=1}^\infty$ is weakly null. Choosing  $0<c<1$, from \cite[Corollary 2.3.5]{meyer} there exists a disjoint sequence $(y_i)_{i=1}^\infty$ formed by positive vectors of $E$ and a subsequence $(z_{k_i})_{i=1}^\infty$ of $(z_k)_{k=1}^\infty$ such that $y_i\leq z_{k_i}$ and $\|y_i\|\geq c$ for every $i\in\mathbb{N}$. In particular the sequence $(y_i)_{i=1}^\infty$ is not norm null. Given $P\in\mathcal{P}^r(^nE)$, $n\in\mathbb{N}$, there are positive $n$-homogeneous polynomials $P_1$ and $P_2$ on $E$ such that $P=P_1-P_2$. By \cite[Proposition 5]{loane1} we know that $P_1$ and $P_2$ are monotone on the positive cone of $E$, that is $P(x) \leq P(y)$ whenever $0 \leq x \leq y$. Therefore,
$$
 0\leq P_m(y_i)\leq P_m(z_{k_i})\longrightarrow 0, ~m=1,2,
$$
from which we get
$$
\|P(y_i)\|=\|P_1(y_i)-P_2(y_i)\|\leq \|P_1(y_i)\|+\|P_2(y_i)\|\longrightarrow 0.
$$
The assumption gives that the sequence $(y_i)_{i=1}^\infty$ is norm null and this contradiction completes the proof.
\end{proof}

According to \cite{botelho}, a Banach lattice $F$ is {\it positively isomorphic} to a subspace of the Banach lattice $E$ if there exists a positive operator from $F$ to $E$ that is an into isomorphism (in the sense of Banach spaces).

\begin{proposition}\label{propsub} {\rm (a)} If the Banach lattice $F$ is positively isomorphic to a subspace of a PPS lattice, then $F$ is a PPS lattice.\\
{\rm (b)} A Banach lattice that is lattice isomorphic to a PPS lattice is a PPS lattice as well.\\
{\rm (c)} Sublattices of PPS lattices are PPS lattices.
\end{proposition}

\begin{proof} (a) Let $E$ be a PPS lattice and $T \colon F \longrightarrow E$ be a positive into isomorphism. Given a positive sequence $(x_j)_{j=1}^\infty$ in $F$ such that $Q(x_j)\longrightarrow 0$ for all $Q\in\mathcal{P}^r(^nF)$ and $n\in\mathbb{N}$, we have $(P\circ T)(x_j)\longrightarrow 0$ for all $P\in\mathcal{P}^r(^nE)$ and $n\in\mathbb{N}$, because $(P\circ T)\in\mathcal{P}^r(^nF)$. Since $(T(x_j))_{j=1}^\infty$ is a positive sequence, the PPS property of $E$ gives $T(x_j)\longrightarrow 0$. Hence  $x_j=T^{-1}\circ T(x_j)\longrightarrow 0$ once $T^{-1} \colon T(F) \longrightarrow F$ is a bounded operator.\\
(b) and (c) follow from (a).
\end{proof}

\begin{example}\rm Here we see a space that both it and its dual are not PPS lattices. The Schreier space $S$ is a Banach lattice with the order given by its  1-unconditional basis (see \cite[Proposition 0.4(iii)]{casazza}). In  \cite[Proposition 0.7]{casazza} we can see that $S$ contains an isometric copy of  $c_0$, then it contains an isometric lattice copy of $c_0$ \cite[Theorem 4.61]{aliprantis2}. Example \ref{exAM} and Proposition \ref{propsub} yield that $S$ is not a PPS lattice.

Consider the dual $S^*$ of the Schreier space as a Banach lattice with its dual structure. The sequence $(e_j)_{j=1}^\infty$ formed by the canonical unit vectors of sequence spaces is polynomially null in $S^*$ (see the proof of \cite[Theorem 3.2]{castillo2}), therefore  $P(e_j)\longrightarrow 0$ for all $P\in \mathcal{P}^r(^nS^*)$ and $n\in\mathbb{N}$. It follows that $S^*$ is not a PPS lattice either.
\end{example}

Now we show that the $\ell_p$-sum of finitely many PPS Banach lattices is a PPS lattice.

\begin{proposition} Let $E$ and $F$ be PPS lattices. Then $E\times F$ is PPS lattice with the coordinatewise order and any lattice norm that makes it a Banach lattice. In particular, $E \oplus_p F$ is a PPS lattice for any $1 \leq p \leq \infty$.
\end{proposition}

\begin{proof} Let $\|\cdot\|_{E\times F}$ be a norm on $E \times F$ that makes it a Banach lattice. The coordinatewise order guarantees that the projections
$$\pi_E\colon E\times F\longrightarrow E~,~\pi_E(x,y)=x,$$
$$\pi_F\colon E\times F\longrightarrow F~,~\pi_F(x,y)=y,$$
are positive linear operators, hence continuous. Let $(x_j,y_j)_{j=1}^\infty$ be a positive sequence in $E\times F$ such that $P(x_j,y_j)\longrightarrow 0$ for all $P\in\mathcal{P}^r(^nE \times F)$ and $n\in\mathbb{N}$. For any polynomials $Q_E\in\mathcal{P}^r(^n E)$ and $Q_F\in\mathcal{P}^r(^n F)$, $Q_E\circ \pi_E$ and $Q_F\circ \pi_F$ are regular homogeneous polynomials on $E\times F$, therefore $$Q_E(x_j)=(Q_E\circ\pi_E)(x_j,y_j)\longrightarrow 0 {\rm ~and~} Q_F(y_j)=(Q_F\circ\pi_F)(x_j,y_j)\longrightarrow 0.$$
Since $E$ and $F$ are PPS lattices, $x_j\longrightarrow 0$ in $E$ and $y_j\longrightarrow 0$ in $F$. The equivalence of all norms that make a Riesz space a Banach lattice gives a constant $C> 0$ such that $\|\cdot\|_{E\times F} \leq C \|\cdot\|_1$. Then,  $$0 \leq \|(x_j,y_j)_{j=1}^\infty\|_{E \times F}\leq C\|(x_j,y_j)_{j=1}^\infty\|_1 = C\left(\|x_j\|_E+\|y_j\|_F\right)\longrightarrow 0.$$
\end{proof}

Next we show that PPS lattices enjoy some nice properties.

\begin{proposition}\label{proorp} Every PPS lattice is a KB-space, hence has order continuous norm, is weakly sequentially complete and is Dedekind complete.
\end{proposition}

\begin{proof} Let $E$ be a PPS lattice. Combining Example \ref{exAM} and Proposition \ref{propsub}(b) and (c) it follows that $E$ does not contain a lattice copy of $c_0$. By \cite[Theorem 4.60]{aliprantis2} we get that $E$ is a KB-space. For the other properties, see, respectively, \cite[p.\,232]{aliprantis2}, \cite[Theorem 4.60]{aliprantis2}, \cite[Corollary 4.10]{aliprantis2}.
\end{proof}

Let us see that the converse of the proposition above is not true, that is, KB-spaces are not always PPS lattices. 

\begin{example}\label{exexex}\rm By $T^*$ we denote Tsirelson's original space \cite{casazza, tsirelson}. We shall use the following features of $T^*$: it is reflexive and has an 1-inconditional basis \cite[p.\,16]{tsirelson}, continuous homogeneous polynomials on $T^*$ are weakly sequentially continuous \cite[p.\,121]{dineen}. The inconditional basis makes $T^*$ a Banach lattice. The reflexivity yields that $T^*$ contains no copy of $c_0$, hence it is a KB-space, 
and that $T^*$ fails the positive Schur property (reflexive spaces contain no copy of $\ell_1$ and Banach lattices with the positive Schur property contain a copy of $\ell_1$ (see \cite[Corollary, p.\,19]{wnuk}). Combining the failure of the positive Schur property with the weak sequential continuity of polynomials we conclude, as we did in the proof of Proposition \ref{fpro}, that $T^*$ is not a PPS lattice.
\end{example}

Thus far we have provided more counterexamples than examples. Next section will fix this situation.

\section{$L_p(\mu)$-spaces}

The purpose of this section is to prove that all $L_p(\mu)$-spaces, $1 \leq p < \infty$, are PPS lattices. This gives, in particular, plenty of examples of PPS lattices failing the positive Schur property, because, as we have just seen in Example \ref{exexex}, infinite dimensional reflexive Banach lattices fail the positive Schur property. Though some of the preparatory results below might be known to experts in Measure Theory, we include short proofs because we found no reference to quote.

The next lemma was inspired in \cite[Theorem 13.25]{aliprantis}, the case of Banach spaces can be found in \cite[p.\,14]{johnson} or \cite[p.\,34]{fabian}.
\begin{lemma}\label{isomedfin} Let $1\leq p< \infty$ and a $\sigma$-finite measure space $(\Omega, \Sigma,\mu)$ be given. Then there exists a probability measure $\lambda$ on $(\Omega, \Sigma)$ such that $L_p(\mu)$ is lattice isometric to $L_p(\lambda)$.
\end{lemma}

\begin{proof} Choose a measurable $\mu$-almost everywhere positive function $g$ such that $\int_\Omega gd\mu=1$ and consider the probability measure
$$\lambda\colon\Sigma\longrightarrow [0,1]~,~\lambda(A)=\int_Agd\mu.$$
It is easy to check that 
the $\mu$-null sets and $\lambda$-null sets coincide, and from the case of Banach spaces it is known that the operator
$$T\colon L_p(\mu)\longrightarrow L_p(\lambda)~, ~T(f)=fg^{-1/p},$$
is an isometric isomorphism. 
We just have to check that $T$ is a Riesz homomorphism: 
given $f_1,f_2\in L_p(\mu)$, since $g^{-1/p}(x)\geq 0$ $\lambda$-almost everywhere,  
\begin{align*}
    T(f_1\lor f_2)(x)
    &=((f_1\lor f_2)g^{-1/p})(x)=(f_1\lor f_2)(x)g^{-1/p}(x)=(f_1(x)\lor f_2(x))g^{-1/p}(x)\\
    &=(f_1(x)g^{-1/p}(x))\lor(f_2(x)g^{-1/p}(x)=(f_1g^{-1/p})(x)\lor(f_2g^{-1/p})(x)\\
    &=T(f_1)(x)\lor T(f_2)(x) = (T(f_1)\lor T(f_2))(x),
\end{align*}
holds $\lambda$-almost everywhere. 
\end{proof}

Let $(\Omega, \Sigma, \mu)$ be a measure space and $A\in\Sigma$ be a non-void measurable set. Then $(A,\Sigma|_A,\mu|_A)$ is a measure space, where  $\Sigma|_A:=\{B\subset A:B\in\Sigma\}$ and $\mu|_A$ is defined on $\Sigma|_A$ in the obvious way (see, e.g., \cite[p.\,2 and p.\,68]{bauer}). For every  $\mu$-integrable function $f$ on $\Omega$, denoting its restriction to $A$ by $f|_A$, it holds
\begin{equation}\label{inte}\int_A(f|_A)d(\mu|_A)=\int_Afd\mu
\end{equation}
 \cite[Lemma 12.5]{bauer}. For $1 \leq p < \infty$, this allows us to regard $L_p(\mu|_A)$ as a subspace of $L_p(\mu)$ via the correspondence 
 $g \in L_p(\mu|_A) \mapsto g_A \in L_p(\mu)$, where $g_A(x)=g(x)$ if $x\in A$ and $g_A(x)=0$ if $x\in(\Omega\setminus A$).  

\begin{lemma}\label{projposcont} The correspondence $f \mapsto f|_A$  is a norm one positive projection from $L_p(\mu)$ onto $L_p(\mu|_A)$.
\end{lemma}

\begin{proof} We just prove that the correspondence is positive (the norm conditon follows from (\ref{inte})). Given a positive function $f\in L_p(\mu)$, there exists a measurable set $B\in\Sigma$ such that $\mu(B)=0$ and $f(x)\geq 0$ for every $x\in \Omega \setminus B$. Then $B\cap A\in\Sigma|_A$,
$$0\leq \mu|_A(B\cap A)=\mu(B\cap A)\leq \mu(B)=0$$
and $f|_A(x)\geq 0$ for every $x\in \Omega\setminus (B\cap A)$. So, $f|_A$ is positive in  $L_p(\mu|_A)$.
\end{proof}

We shall actually prove that $L_p(\mu)$-spaces enjoy a property stronger than being a PPS lattice.

\begin{definition}\rm Given $n \in \mathbb{N}$, a Banach lattice $E$ has the {\it $n$-positive polynomial Schur property} if every weakly null sequence $(x_j)_{j=1}^\infty$ in $E$ formed by positive vectors such that $P(x_j) \longrightarrow 0$ for every regular $n$-homogeneous polynomial $P$ on $E$ is norm null. In this case we say that $E$ is an {\it $n$-PPS lattice}.
\end{definition}

The next lemma is actually a first step in the proof that arbitrary $L_p(\mu)$-spaces are $n$-PPS lattices for each $n \geq p$.

\begin{lemma}\label{Lpmedfin} Let $ 1\leq p<\infty$ and $n \in \mathbb{N}$.
If $L_p(\lambda$) is an $n$-PPS lattice for every probability measure $\lambda$, then $L_p(\mu)$ is an $n$-PPS lattice for any measure $\mu$.
\end{lemma}

\begin{proof} Let $(\Omega, \Sigma, \mu)$ be a measure space and $(f_j)_{j=1}^\infty$ be a positive weakly null sequence in $L_p(\mu)$ such that $P(f_j)\longrightarrow 0$ for every $P\in\mathcal{P}^r(^nL_p(\mu))$. For $j,k\in\mathbb{N}$ consider the measurable set $A_{j,k}=\{x\in\Omega:|f_j(x)|>1/k\}$. Since $f_j\in L_p(\mu)$,   $A_{j,k}$ has finite measure for all $j,k\in\mathbb{N}$. Thus,  $A:=\bigcup\limits_{j,k=1}^\infty A_{j,k}$ is a $\sigma$-finite set and $f_j(x)=0$ for all $x\in\Omega\setminus A$ and $j\in\mathbb{N}$.

 Considering the $\sigma$-finite measure space $(A,\Sigma|_A,\mu|_A)$, Lemma \ref{isomedfin} gives that $L_p(\mu|_A)$ is lattice isometric to $L_p(\lambda)$ for some probability measure $\lambda$, which is an $n$-PPS lattice by assumption. So $L_p(\mu|_A)$ is an $n$-PPS lattice by the analogue of Proposition \ref{propsub} to $n$-PPS lattices. Call $T \colon L_p(\mu)\longrightarrow L_p(\mu|_A)$ the norm one positive projection provided by Lemma \ref{projposcont}. In this fashion, $(T(f_j))_{j=1}^\infty$ is a positive sequence in $ L_p(\mu|_A)$ and, for any $Q\in\mathcal{P}^r(^nL_p(\mu|_A))$ we have $(Q\circ T)\in\mathcal{P}^r(^nL_p(\mu))$, so  $Q(T(f_j))\longrightarrow 0$. The $n$-PPS property of $L_p(\mu|_A)$ implies   $\|f_j\|_p = \|T(f_j)\|_p\longrightarrow 0$, proving that 
 $L_p(\mu)$ is an $n$-PPS lattice.
\end{proof}

The last ingredient we need can be found, essentially, in the middle of the proof of \cite[Theorem 4.55]{fabian}.

\begin{lemma}\label{subseqconjdisj} Let $(\Omega,\Sigma, \mu)$ be a probability space, $1\leq p<\infty$ and let $(f_j)_{j=1}^\infty$ be a sequence in $L_p(\mu)$  that converges to zero in measure. Then there exists a subsequence $(f_{j_k})_{k=1}^\infty$ and a sequence $(A_k)_{k=1}^\infty$ of pairwise disjoint measurable sets such that $\|f_{j_k}-f_{j_k}\chi_{A_k}\|_p\longrightarrow 0$.
\end{lemma}

\begin{proof} 
    Put $n_1 = 1$, that is, $f_{n_1}=f_1$, and consider the set $F_1=\{x\in\Omega : |f_{n_1}(x)|^p>1/2\}$. Since $f_{n_1}\in L_p(\mu)$, there exists $\delta_1>0$ such that $\int_E|f_{n_1}|^pd\mu<1/2$ whenever $\mu(E)<\delta_1$. The convergence in measure to zero provides an integer $n_2>n_1$ such that $\mu\{x\in\Omega : |f_{n_2}(x)|^p>1/2^2\}<\delta_1$. Define $F_2=\{x\in\Omega : |f_{n_2}(x)|^p>1/2^2\}$ and note that, as $f_{n_1},~f_{n_2}\in L_p(\mu)$, there exists $\delta_2>0$ such that $\int_E|f_{n_i}|^pd\mu<1/2^2$, $i=1,2$, whenever $\mu(E)<\delta_2$. Using again the convergence in measure to zero, there exists $n_3>n_2$ such that $\mu\{x\in\Omega : |f_{n_3}(x)|^p>1/2^3\}<\delta_2$. Define $F_3=\{x\in\Omega: |f_{n_3}(x)|^p>1/2^3\}$ and note that, as $f_{n_1},~f_{n_2},~f_{n_3}\in L_p(\mu)$, there exists $\delta_3>0$ such that $\int_E|f_{n_i}|^pd\mu<1/2^3$, $i=1,2,3$, whenever $\mu(E)<\delta_3$. Proceeding in this way we construct a subsequence $(f_{n_k})_{k=1}^\infty$ and a sequence $\{F_k\}_{k=1}^\infty$ of measurable sets such that
\begin{align*}
\|f_{n_k}-f_{n_k}\chi_{F_k}\|_p&=\left(\int_{\Omega\setminus F_k}|f_{n_k}|^p d\mu\right)^{1/p}<\left(\int_{\Omega\setminus F_k}1/2^kd\mu\right)^{1/p}\leq(1/2^k)^{1/p}
\end{align*}
for every $k\in\mathbb{N}$. Now we define
 $$\textstyle A_1=F_1\setminus\bigcup\limits_{k>1}F_k, A_2=F_2\setminus\bigcup\limits_{k>2}F_k, \cdots, A_j=F_j\setminus\bigcup\limits_{k>j}F_k, \ldots$$ It is clear that these are pairwise disjoint measurable sets. 
For a fixed $k\in\mathbb{N}$, consider the measure $\lambda(A)=\int_A|f_{n_k}|^pd\mu$, $A \in \Sigma$. Taking into account the definition of the sets $F_k$, $k \in \mathbb{N}$, more precisely the fact that $\int_{F_j}|f_{n_k}|^p d\mu< \frac{1}{2^{j-1}}$ for $j > k$ , we have
\begin{align*}
    \|f_{n_k}\chi_{F_k}-f_{n_k}\chi_{A_k}\|_p^p
    &=\int_\Omega|f_{n_k}(\chi_{F_k}-\chi_{A_k})|^pd\mu=\int_\Omega |f_{n_k}\chi_{(F_k\cap(\bigcup_{j>k}F_j))}|^pd\mu\\
    &=\lambda\Big(F_k\cap\Big({\textstyle\bigcup\limits_{j>k}}F_j\Big)\Big)\leq\lambda
    \Big({\textstyle\bigcup\limits_{j>k}}F_j\Big)\leq\sum_{j>k}\lambda(F_j)\\
    &=\sum_{j>k}\int_{F_j}|f_{n_k}|^pd\mu <\sum_{j>k}
    \frac{1}{2^{j-1}}=\frac{1}{2^{k-1}}.
    \end{align*}
it follows that
$$
\|f_{n_k}-f_{n_k}\chi_{A_k}\|_p\leq\|f_{n_k}-f_{n_k}\chi_{F_k}\|_p+\|f_{n_k}\chi_{F_k}-f_{n_k}\chi_{A_k}\|_p <
\frac{1}{2^{k/p}}+\frac{1}{2^{(k-1)/p}},
$$
from which we get the desired convergence $\|f_{n_k}-f_{n_k}\chi_{A_k}\|_p\longrightarrow 0$.
\end{proof}

\begin{theorem}\label{Lp} For all measures $\mu$ and $1 \leq p  <\infty$, $L_p(\mu)$ is an $n$-PPS lattice for every $n \geq p$.
\end{theorem}

\begin{proof} By Proposition \ref{Lpmedfin} we can restrict ourselves to a probability measure $\mu$. Let $n \geq p$ be given and let $(f_m)_{m=1}^\infty$ be a positive weakly null sequence in $L_p(\mu)$ such that $P(f_m)\longrightarrow 0$ for every $P\in\mathcal{P}^r(^nL_p(\mu))$.
Consider an arbitrary subsequence $(f_j)_{j=1}^\infty$ of $(f_m)_{m=1}^\infty$. Since $(f_j)_{j=1}^\infty$ is weakly null in $L_p(\mu)$, it is bounded, and we can suppose that $\|f_j\|_p\leq 1$ for every $j\in\mathbb{N}$.
As $\mu$ is a probability measure, $L_p(\mu)\hookrightarrow L_1(\mu)$ is a norm one inclusion, so $\|f_j\|_1\leq 1$ for every $j\in\mathbb{N}$ and  $f_j\stackrel{w}{\longrightarrow} 0$ in $L_1(\mu)$ (bounded linear operators are weak-to-weak continuous). It follows that $\|f_j\|_1\longrightarrow 0$ because $L_1(\mu)$ has the positive Schur property.
Convergence in $L_1(\mu)$ implies convergence in measure, so Lemma \ref{subseqconjdisj} provides a subsequence $(f_{j_k})_{k=1}^\infty$ of $(f_j)_{j=1}^\infty$ and pairwise disjoint measurable sets $(A_k)_{k=1}^\infty$  so that
\begin{equation}\label{nneeww}\|f_{j_k}-f_{j_k}\chi_{A_k}\|_p\longrightarrow 0.
\end{equation}
Call $w_k:=f_{j_k}\chi_{A_k}$, that is, $w_k(x) = f_{j_k}(x)$ if $x\in A_k$ and  $w_k(x) =0$ if $x\in\Omega\setminus A_k$.
Let us see that
\begin{align*}
    T\colon L_p(\mu)^n\longrightarrow \mathbb{R}~,~T(g_1, \ldots, g_n ) =\sum_{k\in\mathbb{N}}\left(\prod_{i=1}^n\int_\Omega g_i(x)(w_k(x))^{p-1}d\mu(x)\right),
\end{align*}
is a well defined, positive, symmetric $n$-linear form. Denote by $q$ the conjugate of $p$ and, given $g_1,\ldots,g_n\in L_p(\mu)$, define $g=|g_1|\lor\cdots\lor |g_n|\in L_p(\mu)$ and $h=\frac{1}{\|g\|_p}g\in B_{L_p(\mu)}$. Take into account the following facts: (i)  $((w_k)^{p-1})_{k=1}^\infty\subseteq B_{L_q(\mu)}$ because $(w_k)_{k=1}^\infty\subseteq B_{L_p(\mu)}$, (ii) H\"older's inequality, (iii) $\|g_i\chi_{A_k}\|_p\leq\|g\|_p\cdot\|h\chi_{A_k}\|_p$ para all $k\in\mathbb{N}$ and $1\leq i\leq n$, which follows from
$$|g_i\chi_{A_k}|\leq g\chi_{A_k}= \|g\|_ph\chi_{A_k},$$
(iv) $\int_{A_k}(h(x))^pd\mu\leq 1$ for every $k\in\mathbb{N}$ and $n\geq p$, (v) the sets $(A_k)_{k=1}^\infty$ are pairwise disjoint and $h\in B_{L_p(\mu)}$. Therefore,
\begin{align*}
    \sum_{k\in\mathbb{N}}\left|\prod_{i=1}^n\int_\Omega g_i(x)(w_k(x))^{p-1}d\mu(x)\right|
    &\leq \sum_{k\in\mathbb{N}}\left(\prod_{i=1}^n\int_\Omega\big|g_i(x)(w_k(x))^{p-1}\big|d\mu(x)\right)\\
    &\leq \sum_{k\in\mathbb{N}}\left(\prod_{i=1}^n\|g_i\chi_{A_k}\|_p\cdot\|(w_k)^{p-1}\|_q\right)\\
    &\leq\sum_{k\in\mathbb{N}}\left(\prod_{i=1}^n\|g_i\chi_{A_k}\|_p\right)\leq \sum_{k\in\mathbb{N}}\left(\prod_{i=1}^n\|g\|_p\cdot\|h\chi_{A_k}\|_p\right)\\
&=\|g\|_p^n\cdot \sum_{k\in\mathbb{N}}\left(\int_{A_k}(h(x))^pd\mu\right)^{n/p}\\
&\leq\|g\|_p^n\cdot \sum_{k\in\mathbb{N}}\left(\int_{A_k}(h(x))^pd\mu\right)
\leq\|g\|_p^n.
    \end{align*}
This proves that $T$ is well defined. The $n$-linearity follows from the fact that,  given functions $g_1, \ldots, g_n, f \in L_p(\mu)$ and a scalar $\lambda$, both of the series
  $$\sum_{k\in\mathbb{N}}\left(\prod_{i=1}^n\left(\int_\Omega g_i(x)(w_k(x))^{p-1}d\mu(x)\right)\right)$$ and
  $$\sum_{k\in\mathbb{N}}\left(\left(\int_\Omega f(x)(w_k(x))^{p-1}d\mu(x)\right)\cdot\prod_{i=2}^n\left(\int_\Omega g_i(x)(w_k(x))^{p-1}d\mu(x)\right)\right)  $$
are convergent. We omit the details and the proofs that $T$ is positive and symmetric. 
Thus, the $n$-homogeneous polynomial associated to $T$,
\begin{align*}
   P(g) = \sum_{k\in\mathbb{N}}\left(\int_\Omega g(x)(w_k(x))^{p-1}d\mu\right)^n {\rm ~for~every~} g \in L_p(\mu),
\end{align*}
is positive, hence monotone in the positive cone of $ L_p(\mu)$ \cite[Proposition 5]{loane1}. As $0\leq w_k\leq f_{j_k}$, we have $0\leq P(w_k)\leq P(f_{j_k})$ for every $k\in\mathbb{N}$. By the assumption on the sequence $(f_m)_{m=1}^\infty$, $P(f_{j_k})\longrightarrow 0$, so
\begin{align*}
 \|w_{k}\|_p^{pn}= \left(\int_\Omega w_{k}(x)(w_{k}(x))^{p-1}d\mu(x)\right)^n   =\sum_{i\in\mathbb{N}}\left(\int_\Omega w_{k}(x)(w_i(x))^{p-1}d\mu\right)^n= P(w_{k}) \longrightarrow 0.
\end{align*}
Combining $\|w_k\|_p \longrightarrow 0$ with (\ref{nneeww}) we get
$$\|f_{j_k}\|_p \leq \|f_{j_k}- w_k\|_p + \|w_k\|_p \longrightarrow 0. $$
This proves that every subsequence of $(f_m)_{m=1}^\infty$ admits a norm null subsequence, which is enough for us to conclude that $(f_m)_{m=1}^\infty$ is norm null.
\end{proof}

Recall that every Hilbert lattice is lattice isometric to some $L_2(\mu)$ \cite[Corollary 2.7.5]{meyer}.

\begin{corollary} Hilbert lattices and $\ell_p$, $1 \leq p < \infty$, are PPS lattices.
\end{corollary}

It is worth mentioning that, although we were inspired and deeply influenced by \cite{carne} and \cite{jaramillo}, apart from the obvious specificities of the lattice setting, the strategy we used in this section differs from theirs in two points. First, it was proved in \cite{carne} that $L_p(\mu)$-spaces, $2\leq p < \infty$, are polynomially Schur, and the case $1 < p < 2$ was settled only when Jaramillo and Prieto \cite{jaramillo} established that superreflexive Banach spaces are polynomially Schur. Here we settled the range $1 \leq p < \infty$ at once without using superreflexivity. Second, the authors of \cite{carne} proved first that $\ell_p$, $1 \leq p < \infty$, is polynomially Schur, then used this fact to conclude that Hilbert spaces, in particular, $L_2(\mu)$-spaces, are polynomially Schur, and finally they conclude that  $L_p(\mu)$-spaces, $2\leq p < \infty$, are polynomially Schur. We did not need to go through these intermediary steps, and the fact that Hilbert lattices and $\ell_p$-spaces are PPS lattices is a consequence of our result. 

We finish this section showing that the positive polynomial Schur property is not preserved by the formation of positive Fremlin projective tensor products. By $E_1\widehat{\otimes}_{|\pi|}\cdots \widehat{\otimes}_{|\pi|} E_n$ we denote the positive Fremlin projective tensor product of the Banach lattices $E_1, \ldots, E_n$ \cite{fremlin1}.

\begin{example}\rm  It was proved by Fremlin in \cite[Example 4C]{fremlin1} that the Banach lattice $L_2[0,1]\widehat{\otimes}_{|\pi|}L_2[0,1]$ is not Dedekind complete. Thus, $L_2[0,1]$ is a PPS lattice (Theorem \ref{Lp}) whereas $L_2[0,1]\widehat{\otimes}_{|\pi|}L_2[0,1]$ is not (Proposition \ref{proorp}).
\end{example}

\section{The weak Dunford-Pettis property}

The main result we prove in this section has a twofold motivation. On the one hand, it is the lattice counterpart of the following equivalence: a Banach space has the Schur property if and only if it has the Dunford-Pettis property and is polynomially Schur \cite[Theorem 7.5]{carne}. On the other hand, in \cite[p.\,3]{chen} the following question is attributed to Wnuk: does every KB-space with the weak Dunford-Pettis property have the positive Schur property? Of course Wnuk was speculating about a property that should be added to the weak Dunford-Pettis property to guarantee the positive Schur property. It is shown in this section that the positive polynomial Schur property does the job. Furthermore, we prove a result on the existence of weakly null non-norm sequences in the positive projective tensor product formed by positive elementary tensors and establish that a Lorentz space and its predual fail the weak Dunford-Pettis property.

Recall that a Banach lattice $E$ has the {\it weak Dunford-Pettis property} (w-DPP) if every weakly compact operator from $E$ into any Banach space maps weakly null sequences in $E$ formed by pairwise disjoint vectors to norm null sequences \cite{leung}.

 The next two results are lattices versions of \cite[Lemma 7.4]{carne} and \cite[Proposition 2.34]{dineen}. The first one is motivated by the fact that the tensor product of weakly null sequences is not necessarily weakly null. We give an example in the lattice setting showing that even the tensor product of polynomially null sequences may fail to be weakly null. The example is based on \cite[Theorem 5.5]{castillo2}.

 \begin{example}\label{exlorentz}\rm By $d(w;1)$ we denote the Lorentz sequence space of \cite[Theorem 5.4]{castillo2} and by $d_*(w;1)$ its predual. The sequence $(e_j)_{j=1}^\infty$ of canonical unit vectors is an 1-unconditional basis for $d(w;1)$ \cite[p.\,1643]{albiac1} and the sequence of coordinate functionals $(e^*_j)_{j=1}^\infty$ is an unconditional basis for $d_*(w;1)$ \cite[p.\,1202]{galego}, hence it is an 1-unconditional basis \cite[p.\,19]{lindenstrauss}. We consider $d_*(w;1)$ as a Banach lattice with the order given by the 1-unconditional basis and $d(w;1)$ with its dual structure. By \cite[Theorem 5.4]{castillo2}, $P(e_j^*) \longrightarrow 0$ and $Q(e_j) \longrightarrow 0$ for all continuous homogeneous polynomials $P$ an $Q$ on $d_*(w;1)$ and $d(w;1)$, respectively; in particular both sequences are weakly null. 
Let us see that the sequence  $(e_j\otimes e^*_j)_{j=1}^\infty$ fails to be weakly null in $d(w;1)\widehat{\otimes}_{|\pi|}d_*(w;1)$. From \cite[Theorem 3.2, p.\,204]{schaefer2} we know that $(d(w;1)\widehat{\otimes}_{|\pi|}d_*(w;1))^*$ is lattice isometric to the space $\mathcal{L}^r(d(w;1);d(w;1))$ of regular linear operators from $d(w;1)$ to itself. Considering the identity operator $id_{d(w;1)}\in {\cal L}^r(d(w;1);d(w;1))$ we have that $id_{d(w;1)}(e_j)(e^*_j)=1$ for every $j\in\mathbb{N}$, thus there exists a functional  $\varphi\in(d(w;1)\widehat{\otimes}_{|\pi|}d_*(w;1))^*$ such that $1=\varphi(e_j\otimes e^*_j)\not\longrightarrow 0$.
 \end{example}

\begin{lemma}\label{seqwnulproten} Let $E_1,\ldots,E_n$ be Banach lattices, all but at most one with the   w-DPP. If $(x^i_j)_{j=1}^\infty$ is a positive weakly null sequence in $E_i$,  $i=1,\ldots,n$, then the sequence $(x^1_j\otimes\cdots\otimes x^n_j)_{j=1}^\infty$ is weakly null in the positive projective tensor product  $E_1\widehat{\otimes}_{|\pi|}\cdots \widehat{\otimes}_{|\pi|}E_n$.
\end{lemma}

\begin{proof} By the associativity of the Fremlin positive projective tensor product \cite[Corollary 1G]{fremlin1} we can assume that $E_1$ is the one that may lack the w-DPP. We shall proceed by induction on $n$. There is nothing to do in the case $n=1$. Suppose that the result holds for $n-1$, that is, $(x^1_j\otimes\cdots\otimes x^{n-1}_j)_{j=1}^\infty$ is weakly null in $E_1\widehat{\otimes}_{|\pi|}\cdots \widehat{\otimes}_{|\pi|}E_{n-1}$.
Calling $X= E_1\widehat{\otimes}_{|\pi|}\cdots \widehat{\otimes}_{|\pi|}E_{n-1}$, the associativity and \cite[Theorem 3.2, p.\,204]{schaefer2} give that 
$(X\widehat{\otimes}_{|\pi|}E_n)^*$ is lattice isometric to $\mathcal{L}^r(X;E^*_n)$ by means of the obvious isomorphism, meaning that for every $\varphi\in(X\widehat{\otimes}_{|\pi|}E_n)^*$ there exists a regular operator $T_\varphi\in \mathcal{L}^r(X;E^*_n)$ such that
$$
\varphi(x^1\otimes\cdots\otimes x^{n-1}\otimes x^n)=T_\varphi(x^1\otimes\cdots\otimes x^{n-1})(x^n)
$$
for all $x^1\otimes\cdots\otimes x^{n-1}\in E_1\widehat{\otimes}_{|\pi|}\cdots\widehat{\otimes}_{|\pi|}E_{n-1}$ and $x^n\in E_n$. The weak-to-weak continuity of $T_\varphi$ gives that the sequence
$(T_\varphi(x^1_j\otimes\cdots\otimes x^{n-1}_j))_{j=1}^\infty$ is weakly null in $E^*_n$. From the w-DPP of $E_n$, together with the characterizations of the w-DPP proved in \cite[Corollary 2.6]{aqzzouz}, it follows that
$$\varphi(x^1_j\otimes\cdots\otimes x^n_j)= T_\varphi(x^1_j\otimes\cdots\otimes x^{n-1}_j)(x^n_j)\longrightarrow 0,
$$
establishing that $(x^1_j\otimes\cdots\otimes x^n_j)_{j=1}^\infty$ is weakly null in  $E_1\widehat{\otimes}_{|\pi|}\cdots \widehat{\otimes}_{|\pi|}E_n$.
\end{proof}


\begin{proposition} \label{seqfracnulptpp}
Let $(x_j)_{j=1}^\infty$ be a positive weakly null sequence in a Banach lattice $E$ with the w-DPP. Then  $P(x_j) \longrightarrow 0$ for every regular homogeneous polynomial $P$ on $E$.
\end{proposition}

\begin{proof} Given $n \in \mathbb{N}$ and $P \in {\cal P}^r(^nE)$, by $\check{P} \in {\cal L}^r(^nE)$ we denote the symmetric regular $n$-linear form associated to $P$. By
\cite[Proposition 3.3]{bu} there is a functional $\varphi \in (\widehat{\otimes}_{n,|\pi|}E)^*$ such that $$\varphi(x_1 \otimes \cdots \otimes x_n) = \check{P}(x_1, \ldots, x_n)$$
for all $x_1, \ldots, x_n \in E$. By Lemma \ref{seqwnulproten} we know that the sequence $(\otimes^n x_j)_{j=1}^\infty$ is weakly null in $\widehat{\otimes}_{n,|\pi|}E$, so
$$P(x_j) = \check{P}(x_j, \ldots, x_j) = \varphi(\otimes^nx_j) \longrightarrow 0. $$
\end{proof}

The main result of this section reads as follows.

\begin{theorem}\label{lth} A Banach lattice has the positive Schur property if and only if it has the weak Dunford-Pettis and the positive polynomial Schur properties.
\end{theorem}

\begin{proof} Suppose that $E$ has the positive Schur property. Using that weakly null sequences are bounded, the characterizations proved in \cite[Corollary 2.6]{aqzzouz} imply immediately that $E$ has the w-DPP. By Example \ref{L1Plam}(a) we know that $E$ is a PPS lattice.

Conversely, suppose that $E$ is a PPS lattice with the w-DPP and let  $(x_j)_{j=1}^\infty E$ be a positive weakly null sequence in $E$. 
Given $n\in \mathbb{N}$ and $P\in\mathcal{P}^r(^nE)$, calling on Proposition \ref{seqfracnulptpp} we have $P(x_j)\longrightarrow 0$. From the PPS property of $E$ it follows that $x_j\longrightarrow 0$. 
\end{proof}

Now we derive a few more consequences of Lemma \ref{seqwnulproten} and of Proposition \ref{seqfracnulptpp}. Combining the lemma with Example \ref{exlorentz} we get the following.

\begin{corollary} Both the Lorentz space $d(w;1)$ and its predual $d_*(w;1)$ fail the weak Dunford-Pettis property.
\end{corollary}

Let $E_1, \ldots, E_n$ be Banach lattices lacking the positive Schur property. It is easy to see that the positive projective tensor product $E_1\widehat{\otimes}_{|\pi|}\cdots \widehat{\otimes}_{|\pi|}E_n$ fails the positive Schur property (see \cite{botelho}). So, there exists a positive non-norm null weakly null sequence in $E_1\widehat{\otimes}_{|\pi|}\cdots \widehat{\otimes}_{|\pi|}E_n$. For practical purposes, it would be very helpful if we could suppose that this sequence lies in the Segre cone of $E_1 \otimes \cdots \otimes E_n$, that is, if it is of the form $(x^1_j\otimes\cdots\otimes x^n_j)_{j=1}^\infty$, $x_j^i \in E_i$, $i=1,\ldots,n$. Let us see that this is possible if, in addition, the lattices have the w-DPP. In view of Theorem \ref{lth} this is same of asking the lattices to have the w-DPP and not to be PPS lattices.

\begin{proposition}\label{lastp} Let $E_1, \ldots, E_n$ be Banach lattices with the w-DPP and  lacking the positive Schur property. Then there are positive sequences $(x_j^i)_{j=1}^\infty$ in $E_i$, $i= 1,\ldots, n$, such that $(x^1_j\otimes\cdots\otimes x^n_j)_{j=1}^\infty$ is a weakly null non-norm null sequence in $E_1\widehat{\otimes}_{|\pi|}\cdots \widehat{\otimes}_{|\pi|}E_n$.
\end{proposition}

\begin{proof} Take positive weakly null non-norm null sequences $(y_j^i)_{j=1}^\infty$ in $E_i$, $i= 1,\ldots, n$. There are $\delta > 0$ and subsequences $(x_j^i)_{j=1}^\infty$ of $(y_j^i)_{j=1}^\infty$ such that $\|x_j^i\| \geq \delta$ for all $j \in \mathbb{N}$ and $i = 1, \ldots, n$. Of course, each $(x_j^i)_{j=1}^\infty$ is positive weakly null in $E_i$, so $(x^1_j\otimes\cdots\otimes x^n_j)_{j=1}^\infty$ is weakly null in $E_1\widehat{\otimes}_{|\pi|}\cdots \widehat{\otimes}_{|\pi|}E_n$ by Lemma \ref{seqwnulproten}. This sequence is not norm null because $\|x^1_j\otimes\cdots\otimes x^n_j\|_{|\pi|} = \|x_j^1\| \cdots \|x_j^n\| \geq \delta^n$ for every $j$.
\end{proof}

The associativity of the full tensor product was used twice in the proof of Lemma \ref{seqwnulproten}. We finish the paper using Proposition \ref{seqfracnulptpp} to show that, despite the lack of associativity in the symmetric tensor product, a symmetric version of Proposition \ref{lastp} holds. For the $n$-fold positive symmetric projective tensor product $\widehat{\otimes}^n_{s,|\pi|}E$ of  the Banach lattice $E$, see \cite{bu}.

\begin{proposition} Let $E$ be a Banach lattice with the w-DPP and lacking the positive Schur property.  Then there exists a positive sequence $(x_j)_{j=1}^\infty$ in $E$ such that, for every $n \in \mathbb{N}$, $(\otimes^nx_j)_{j=1}^\infty$ is a weakly null non-norm null sequence in $\widehat{\otimes}^n_{s,|\pi|}E$.
\end{proposition}

\begin{proof} We can consider a positive weakly null non-norm null sequence $(x_j)_{j=1}^\infty$ in $E$ and $\delta > 0$ such that $\|x_j\| \geq \delta$ for every $j$. Given $n \in \mathbb{N}$, Proposition \ref{seqfracnulptpp} gives that $P(x_j) \longrightarrow 0$ for every regular $n$-homogeneous polynomial $P$ on $E$. Given $\varphi \in (\widehat{\otimes}^n_{s,|\pi|}E)^*$, by \cite[Proposition 3.4]{bu} there exists a regular $n$-homogeneous polynomial $P$ on $E$ such that $\varphi(\otimes^n x) = P(x)$ for every $x \in E$. It follows that the sequence $(\otimes^nx_j)_{j=1}^\infty$ is weakly null in $\widehat{\otimes}^n_{s,|\pi|}E$. It is not norm null because $\|\hspace*{-0.15em}\otimes^n x_j\|_{|\pi|}= \|x_j\|^n \geq \delta^n$ for every $j$.
\end{proof}

For the role of the Segre cone in the modern theory of classes of multilinear operators, see \cite{mexicanos1, mexicanos2}.

%
%
%

\bigskip

\noindent Faculdade de Matem\'atica~~~~~~~~~~~~~~~~~~~~~~Departamento de Matem\'atica\\
Universidade Federal de Uberl\^andia~~~~~~~~ IMECC-UNICAMP\\
38.400-902 -- Uberl\^andia -- Brazil~~~~~~~~~~~~ 13.083-859 -- Campinas -- Brazil\\
e-mail: botelho@ufu.br ~~~~~~~~~~~~~~~~~~~~~~~~~e-mail: lucasvt09@hotmail.com

\end{document}